\documentclass[12pt]{article}

\usepackage{amsfonts,amssymb,amsthm, amsmath}
\usepackage{dsfont}
\usepackage{fullpage}
\usepackage{cite}

\DeclareMathOperator{\dist}{dist} 

\DeclareMathOperator{\StSh}{StSh}

\DeclareMathOperator{\OrientSh}{OrientSh}
\DeclareMathOperator{\Rep}{Rep}

\DeclareMathOperator{\Cyl}{Cyl}\DeclareMathOperator{\Sp}{Sp}

\newcommand{\RR}{\ensuremath{\mathds{R}}}

\newcommand{\Cc}{\mbox{${\bf C}$}}
\newcommand{\Ss}{\mbox{$\textbf{S}$}}

\newcommand{\ep}{\varepsilon}
\newcommand{\dd}{d}

\newcommand{\al}{\alpha}

\newcommand{\Ws}{\mbox{$W^s$}}
\newcommand{\Wu}{\mbox{$W^u$}}

\newcommand{\Cone}{\mbox{$\Cc^1$}}

\newcommand{\sref}[1]{(\ref{#1})}

\newtheorem{thm}{Theorem}
\newtheorem{proposition}{Proposition}
\newtheorem{lem}{Lemma}

\theoremstyle{definition}

\newtheorem{remark}{Remark}

\begin{document}

\title{An example of a vector field with the oriented shadowing property}

\author{Sergey Tikhomirov\footnote{Max Plank Institute for Mathematics in the Sciences, Inselstrasse 22, Leipzig, 04103, Germany; Chebyshev Laboratory, Saint-Petersburg State Univeristy, 14th line of Vasilievsky island, 29B, Saint-Petersburg, 199178, Russia;
sergey.tikhomirov@gmail.com} \footnote{partially supported by Chebyshev Laboratory (Department of Mathematics and Mechanics, St. Petersburg State University)  under RF Government grant 11.G34.31.0026, JSC ``Gazprom neft'', by the St. Petersburg State University in the framework of project 6.38.223.2014 and  by the German-Russian
Interdisciplinary Science Center (G-RISC) funded by the German Federal
Foreign Office via the German Academic Exchange Service (DAAD)} }

\date{}

\maketitle

\begin{abstract}
We consider shadowing properties for vector fields corresponding to different type of reparametrisations. We give an example of a vector field which has the oriented shadowing properties, but does not have the standard shadowing property.
\end{abstract}
\textbf{keywords:} shadowing, vector field, reparametrization, structural stability.\newline

\section{Introduction}

The theory of shadowing of approximate trajectories (pseudotrajectories) of dynamical systems is now a well developed part of the global theory of dynamical systems (see, for example, the monographs \cite{PalmBook, PilBook}). Often shadowing is also called pseudo orbit tracing property (POTP).

This theory is closely related to the classical theory of structural stability (the basic definitions of structural stability and $\Omega$-stability for diffeomorphisms and vector fields can be found, for example, in the monograph \cite{PilSSBook, Katok}). It is well known that diffeomorphisms and vector fields have shadowing property in a neighborhood of a hyperbolic set \cite{Ano, Bow} and structurally stable systems (satisfying Axioma A and the strong transversality condition) have shadowing property on the
whole manifold \cite{Rob, Saw, PilFlow}.

Let us note that the main difference between the shadowing problem for vector fields and the similar problem for diffeomorphisms is related to the necessity of reparametrization of shadowing trajectories in the former case. Several definitions of shadowing property are possible \cite{Kom, Tho} according to the choice of reparametrization of shadowing orbit. In the modern theory of shadowing \cite{PilRev} the most important types of reparametrisations corresponds to standard and oriented shadowing properties (strong POTP and normal POTP in the sense of \cite{Kom}).

Komuro proved that oriented and standard shadowing properties are equivalent for vector fields without fixed points \cite{Kom}. In the same paper it was posed a question if those two notions are different in general \cite[Remark 5.1]{Kom}. Up to our knowledge the answer to this question is still unknown.

Recently importance of this question appears again, during characterisation of vector fields with the $C^1$-robust shadowing properties \cite{LeeSak,TikhVest,PilTikhDAN, PilTikh, GanLiTikh}. In particular in a work by the author \cite{PilTikh} it was constructed an example of a not structurally
stably vector field with the $C^1$-robust oriented shadowing property. It looks like the statement is correct for the standard shadowing property, however the proof does not work in this case.

In the present paper we give an example of a vector field on a $4$-dimensional manifold which has the oriented shadowing property and do not have the standard shadowing property. An example is a vector field with a nontransverse intersection of stable and unstable manifolds of two fixed points of a very special structure in their neighborhoods.

The paper is organised as the following: in Section \ref{secDef} we give all necessarily definitions and formulate the main result; in Section \ref{secaux} we construct a special two-dimensional vector field which plays a central role in the construction of the example; in Section \ref{secconstr} we make the construction of the example; in Section \ref{secorient} we prove that example has the oriented shadowing property; in Section \ref{secst} we prove that example does not have the standard shadowing property; in the Appendix we prove properties of above-mentioned two dimensional vector field.

\section{Definitions and main results} \label{secDef}

Let $M$ be a smooth compact manifold with Rimannian metric $\dist$ and let $X$ be a vector field on $M$ of class $C^1$, and $\phi(t, x)$ flow generated by it.

For $a> 0$ and $x \in M$ denote by $B(a, x)$ an open ball centered at $x$.

For $d> 0$ we say that (not necessarily continuous) map $g: \RR \to M$ is a $d$-pseudotrajectory if holds the inequalities
\begin{equation}\label{eqpst}\notag
\dist(g(t+ \tau), \phi(\tau, g(t))) < d, \quad t \in \RR, |\tau| < 1.
\end{equation}

An increasing homeomorphism $h$ of the real line is called a reparametrisation. Denote set of all reparametrisations by $\Rep$. For $a> 0$ we say that a reparametrisation $h$ belongs to class $\Rep(a)$ if
$$
\left| \frac{h(t_1)- h(t_2)}{t_1 - t_2} - 1 \right| < a, \quad t_1, t_2 \in \RR, \; t_1 \ne t_2.
$$

We say that a vector field has \textit{the standard shadowing property} if for any $\ep > 0$ there exists $d > 0$ such that for any $d$-pseudotrajectory $g$ there exists $x_0 \in M$ and a reparametrisation $h \in \Rep(\ep)$ such that
\begin{equation}\label{eqsh}
\dist(g(t), \phi(h(t), x_0) < \ep, \quad t \in \RR.
\end{equation}
In this case we say that $g$ is $\ep$-standard shadowed by  $\phi(h(\cdot), x_0)$. Denote the set of all vector fields with the standard shadowing property by $\StSh$.

We say that a vector field has \textit{the oriented shadowing property} if for any $\ep > 0$ there exists $d > 0$ such that for any $d$-pseudotrajectory $g$ there exists $x_0 \in M$ and a reparametrisation $h \in \Rep$ such that \sref{eqsh} holds.
In this case we say that $g$ is $\ep$-oriented shadowed by $\phi(h(\cdot), x_0)$. Denote the set of all vector fields with the oriented shadowing property by $\OrientSh$.

Clearly $\StSh \subset \OrientSh$.

It is important to mention that structurally stable vector fields have standard and oriented shadowing properties. Moreover the following generalisation is possible.

We say that a vector field has the Lipschitz shadowing property if there exists $L, d_0 > 0$ such that for any $d \in (0, d_0)$ and  $d$-pseudotrajectory $g$ there exists $x_0 \in M$ and a reparametrisation $h \in \Rep(Ld)$ such that inequalities \sref{eqsh} hold for $\ep = Ld$.

\begin{thm}\label{thmSS}\cite{PilFlow}
Structurally stable vector fields satisfy the Lipschitz shadowing property.
\end{thm}
\begin{remark} \label{remNoRep}
In fact, it is shown in \cite{PilFlow} that if a structurally stable
vector field does not have closed trajectories, then it has the Lipschitz shadowing property without
reparametrization of shadowing trajectories: there exists $L>0$ such
that if $g(t)$ is a $d$-pseudotrajectory with small $d$, then there
exists a point $x$ such that
$$
\dist(g(t),\phi(t,x))\leq Ld,\quad t\in\RR.
$$
\end{remark}
\begin{remark}\label{remConv}
It is worth to mention that recently it was proved that converse of Theorem \ref{thmSS} is correct \cite{PalmPilTikh}. Note that the proof strongly uses both conditions $h \in \Rep(Ld)$ and $\ep = Ld$.
\end{remark}

In the paper we prove
\begin{thm}\label{thmMain}
For $M = S^2 \times S^2$ there exists vector field $X \in \OrientSh \setminus \StSh$.
\end{thm}
It is interesting to understand if our example satisfy the Lipschitz oriented shadowing property (there exists $L, d_0 > 0$ such that for any $d < d_0$ and $d$-pseudotajectory $g(t)$ there exists $x_0$ and $h \in \Rep$ such that inequalities \sref{eqsh} hold for $\ep = Ld$). If this statement is correct it would justify that assumption $h \in \Rep(Ld)$  in Remark \ref{remConv} is essential. Unfortunately we cannot prove this (see Remark \ref{remProb}) and leave this out of the scope of the paper.

\section{Auxilarily Statements} \label{secaux}

Consider $a,l > 0$, $K > 3$ and a continuous function $b: [0, +\infty) \to \RR$, $b \in C^1(0, +\infty)$ defined as the following:
$$
b(r) =
\begin{cases}
0, & \quad r \in \{0\} \cup ((K-1)l, +\infty),\\
-\frac{1}{\ln r}, & \quad r \in (0, 2l),\\
b(r) \geq 0, & \quad r
\in [2l, (K-1)l].
\end{cases}
$$
Let $\psi(t, x)$ be a flow on $\RR^2$ generated by a vector field defined by the following formula
\begin{equation}\label{eq2dvf}\notag
Y(x) = \left(
            \begin{array}{cc}
              a & b(|x|) \\
              b(|x|) & a \\
            \end{array}
          \right)x,
\end{equation}
which generates the following system of differential equations in polar coordinates
\begin{equation}\label{eq2dpolar}\notag
\begin{cases}
\frac{\dd r}{\dd t} = ar,\\
\frac{\dd \varphi}{\dd t} = b(r).
\end{cases}
\end{equation}

For a point $x \in \RR^2 \setminus \{0\}$ we denote by
$\arg(x)$ the point $ \frac{x}{|x|} \in \Ss^1$. If a point $x \in \RR^2$
has polar coordinates $(r, \varphi)$, and~$r \ne 0$, we put $\arg(x)
= \varphi$.
\begin{lem}\label{lem2d}
\begin{itemize}
\item[(i)] For any $a, l > 0$, $K > 3$ vector field $Y$ is of class $C^1$.
\item[(ii)] For any $a, l > 0$, $K > 3$ and a point $x_0 \in \RR^2 \setminus {0}$, angle $\Theta$ and  $T_0 < 0$ there exists $t < T_0$ such that $\arg(\psi(t, x_0)) = \Theta$.
\item[(iii)] There exists $a, l > 0$, $K > 3$ such that the following condition holds. If for some points $x_0, x_1 \in \RR^2$, $|x_0| < l$, $|x_1| < 2l$ and reparametrization $h \in \Rep(l)$, holds inequalities
\begin{equation}\label{eq2dstsh}
\dist(\psi(h(t), x_1), \phi(t, x_0)) < l
\end{equation}
provided that $|\psi(h(t), x_1)|,|\phi(t, x_0)| < 1$. Then
$|arg(x_1) - arg(x_0)| < \pi/4$.

\end{itemize}
\end{lem}
The proof of this lemma is quite technical, we give it in the Appendix.
\begin{remark}
Vector field $Y$ is of class $C^1$ but not $C^{1+ H\ddot{o}lder}$. We do not know if it is possible to construct a 2-dimensional vector field of class $C^{1+ H\ddot{o}lder}$ satisfying items (ii), (iii) of Lemma~\ref{lem2d}. As the result our example of vector field $X$ is not $C^{1+ H\ddot{o}lder}$. We do not know if it is an essential restriction or drawback of our particular construction.
\end{remark}

For the rest of the paper let us fix $a, l > 0$, $K > 3$ from item (iii) of Lemma \ref{lem2d}.

We will also need the following statement, which we prove in the appendix.
\begin{lem}\label{lemCyl}
Let $S_1$ and $S_2$ be three-dimensional hyperplanes with coordinates
$(x_1, x_2, x_3)$ and $(y_1, y_2, y_3)$ respectively. Let $Q: S_2 \to S_1$
be a linear map satisfying the following condition
$$
Q\{y_2 = y_3 = 0\} \ne \{x_2 = x_3 = 0\}.
$$
Then for any $D> 0$ there exists $R>0$ (depending on $Q$ and $D$) such that
for any two sets $\Sp_1 \subset S_1 \cap \{x_1 =
0\}$ and $\Sp_2 \subset S_2 \cap \{y_1 = 0\}$ satisfying
\begin{itemize}
\item $\Sp_1 \subset B(R, 0)$, $\Sp_1 \subset B(R, 0)$;
\item $Sp_1$ intersects any halfline in  $S_1 \cap \{x_1 =
0\}$ starting at 0;
\item $Sp_2$ intersects any halfline in  $S_2 \cap \{y_1 =
0\}$ starting at 0;
\end{itemize}
the sets
$$
\Cyl_1 = \{(x_1, x_2, x_3), \quad |x_1| < D, \; (0, x_2, x_3) \in \Sp_1
\},
$$
$$
\Cyl_2 = \{(y_1, y_2, y_3), \quad |y_1| < D, \; (0, y_2, y_3) \in \Sp_2
\}
$$
satisfy the condition $\Cyl_1 \cap Q \Cyl_2 \ne \emptyset$.
\end{lem}

\section{Construction of a 4-dimensional vector field}\label{secconstr}

Consider a vector field $X$ on the manifold $M = S^2 \times S^2$
that has the following properties (F1)-(F6) ($\phi$ denotes the
flow generated by $X$).

\begin{enumerate}
\item[\textbf{(F1)}] The nonwandering set of $\phi$ is the union of
four rest points $p, q, s, u$.
\item[\textbf{(F2)}] In the neighborhoods $U_p = B(1, p)$, $U_q = B(1, q)$ one can introduce coordinates such that
$$
X(x)=J_p(x-p),\quad x \in U_p;\quad
\quad X(x)=J_q(x-q),\quad x \in U_q,
$$
where

\begin{equation}\notag
J_p(x) = \left(
            \begin{array}{cccc}
              -1 & 0 & 0 & 0 \\
              0 & -2 & 0 & 0 \\
              0 & 0 & a & -b(r(x_3, x_4)) \\
              0 & 0 & b(r(x_3, x_4)) & a \\
            \end{array}
          \right)x,
\end{equation}
\begin{equation}\notag
J_q(x) = -\left(
           \begin{array}{cccc}
             -1 & 0 & 0 & 0 \\
             0 & a & 0 & -b(r(x_2, x_4)) \\
             0 & 0 & -2 & 0 \\
             0 & b(r(x_2, x_4)) & 0 & a \\
           \end{array}
         \right)x.
\end{equation}
For point $x = (x_1, x_2, x_3, x_4) \in U_p$ denote $P_1x = x_1$, $P_{34}x  = (x_3, x_4)$, for point $x = (x_1, x_2, x_3, x_4) \in U_q$ denote $P_1x = x_1$, $P_{24}x  = (x_2, x_4)$, etc.

\item[\textbf{(F3)}] The point $s$ is an attracting hyperbolic rest point. The point
$u$ is a repelling hyperbolic rest point. The following condition
holds:
$$
\Wu(p) \setminus \{p\} \subset \Ws(s), \quad \Ws(q)
\setminus \{q\} \subset \Wu(u),
$$
where $W^u(p)$ is the unstable manifold of $p$, $W^s(q)$ is the stable manifold of $q$, etc. For $m> 0$ we denote $W^u_{loc}(p, m)=\Wu(p) \cap B(m, p)$ etc.

\item[\textbf{(F4)}] The intersection of $\Ws(p) \cap \Wu(q)$ consists of a single trajectory $\al$, satisfying the following
    $$
    \al \cap U_p \subset \{p+(t, 0, 0, 0); t \in (0, 1\}; \quad
    \al \cap U_q \subset \{q-(t, 0, 0, 0); t \in (0, 1\}
    $$
\item[\textbf{(F5)}] For some $\Delta \in (0, 1)$, $T_a > 0$ the following holds
$$
\phi(T_a, q + (-1, x_2, x_3, x_4)) = (p + (1, x_2, x_3, x_4)), \quad |x_2|, |x_3|, |x_4| < \Delta.
$$

\item[\textbf{(F6)}]
$\phi(t,x) \notin U_q$, for $x \in U_p, \; t\geq 0$.
\end{enumerate}
The construction is similar to \cite[Appendix A]{PilTikh}. We leave details to the reader.

\begin{thm}\label{thmOrientSh}
Vector field $X$ satisfies the oriented shadowing property.
\end{thm}
\begin{thm}\label{thmStSh}
Vector field $X$ does not satisfy the standard shadowing property.
\end{thm}

Trivially Theorem \ref{thmMain} follows from Theorems \ref{thmOrientSh}, \ref{thmStSh}.

\section{Oriented Shadowing property}\label{secorient}

Fix $\ep > 0$.

For points $y_p = \al(T_p) \in U_p , y_q  = \al(T_q) \in U_q$ (note that $T_p > T_q$) and $\delta > 0$ we say that ${g}(t)$ is a pseudotrajectory of
type \text{Ps}$(\delta)$ if
\begin{equation}\label{Pst1.2.2}\notag
{g}(t) =
\begin{cases}
\phi(t - T_p, x_p), & t > T_p, \\
\phi(t - T_q, x_q), & t < T_q, \\
\alpha(t), & t \in [T_q, T_p],
\end{cases}
\end{equation}
for some points $x_p \in B(\delta, y_p)$ and $x_q \in B(\delta, y_q)$.

\begin{proposition}\label{prop2}
For any $\delta > 0$, $y_p \in  \al\cap
U_p$, and $y_q \in \al\cap U_q$ there exists $d> 0$ such that if
$g(t)$ is a $d$-pseudotrajectory of $X$, then either $g(t)$ can  be
$\ep$-oriented shadowed or there exists a pseudotrajectory $g^*(t)$ of type
\text{Ps}$(\delta)$ with these $y_p$ and $y_q$ such that
$$
\dist(g(t), g^*(t)) < \ep/2,\quad t\in\RR.
$$
\end{proposition}

\begin{proposition}\label{prop3} There exists $\delta > 0$, $y_p
\in\al\cap U_p$, and $y_q \in \al\cap U_q$ such that any
pseudotrajectory of type \text{Ps}$(\delta)$ with these $y_p$
and $y_q$ can be $\ep/2$-oriented shadowed.
\end{proposition}

Clearly, Propositions \ref{prop2} and \ref{prop3} imply that $X \in \OrientSh$.

Proof of Proposition \ref{prop2} is standard. Exactly the same statement was proved in \cite[Proposition 2]{PilTikh} for a slightly different vector field (the only difference is in the structure of matrixes $J_p$, $J_q$). The proof can be literally repeated in our case.

The main idea of the proof is the following. In parts ``far'' from $\al$ vector field is structurally stable and hence have shadowing property according to Remark \ref{remNoRep}. This statement implies that if $g(t)$ does not intersect a small neighborhood of $\al$ it can be shadowed. If $g(t)$ intersects a small neighborhood of $\al$ then (after a shift of time) for  $t > T_p$ points $g(t)$ also lies in a structurally stable part of $X$ and can be shadowed by  $\phi(t - T_p, x_p)$; similarly for $t < T_q$ points $g(t)$ can be shadowed by  $\phi(t - T_q, x_q)$; for $t \in (T_q, T_p)$ points $g(t)$ are close to $\al$. We omit details in the present paper.


\begin{proof}[Proof of Proposition \ref{prop3}]
Without loss of generality, we may assume that
$$
O^+(B(\ep/2, s),\phi) \subset B(\ep, s)\quad\mbox{and} \quad
O^-(B(\ep/2, u),\phi) \subset B(\ep, u).
$$

Take  $m \in (0, \ep/8)$. We take points $y_p=\alpha(T_p) \in B(m/2, p)\cap \al$ and $y_q=\alpha(T_q) \in B(m/2, q) \cap \al$. Put $T=T_p - T_q$. Take $\delta>0$
such that if $g(t)$ is a pseudotrajectory of type
\text{Ps}$(\delta)$ (with $y_p$ and $y_q$ fixed above), $t_0 \in
\RR$, and $x_0 \in B(2\delta, g(t_0))$, then
\begin{equation}\label{Pst1.4.1}
\dist(\phi(t-t_0, x_0), g(t)) < \ep/2 , \quad |t-t_0| \leq T+1.
\end{equation}

Consider a number $\tau > 0$ such that if $x \in \Wu(p) \setminus
B(m/2, p)$, then $\phi(\tau, x) \in B(\ep/8,s)$. Take
$\ep_1\in(0,m/4)$ such that if two points $z_1, z_2 \in M$ satisfy
the inequality $\dist(z_1, z_2) < \ep_1$, then
$$
\dist(\phi(t, z_1), \phi(t, z_2))< \ep/8, \quad |t| \leq \tau.
$$
In this case, for any $y \in B(\ep_1, x)$ the following inequalities hold:
\begin{equation}\label{Pst1.5.1}
\dist(\phi(t, x), \phi(t, y)) < \ep/4, \quad t \geq 0.
\end{equation}
Decreasing  $\ep_1$, we may assume that if $x'\in \Ws(q)
\setminus B(m/2, q)$ and $y' \in B(\ep_1, x')$, then
\begin{equation}\label{Pst1.5.1'}\notag
\dist(\phi(t, x'), \phi(t, y')) < \ep/4, \quad t \leq 0.
\end{equation}

Let $g(t)$ be a pseudotrajectory of type \text{Ps}$(\delta)$, where
$y_p$, $y_q$ and $\delta$ satisfy the above-formulated conditions.

Let us consider several possible cases.
\medskip

Case (P1):  $x_p \notin \Ws(p)$ and  $x_q \notin \Wu(q)$. Let
$$
T' = \inf \{ t \in \RR: \; \phi(t, x_p) \notin B(p, 3m/4) \}.
$$
If $\delta$ is small enough, then $\dist(\phi(T',x_p), \Wu(p)) <
\ep_1$. In this case, there exists a point $z_p \in W^u_{loc}(p,
m)\setminus B(m/2, p)$ such that
\begin{equation}\label{Text8.0.5}
\dist(\phi(T', x_p), z_p) < \ep_1.
\end{equation}

Applying a similar reasoning in a neighborhood of $q$ (and reducing
$\delta$, if necessary), we find a point $z_q \in W^s_{loc}(q,
m)\setminus B(m/2, q)$ and a number $T'' < 0$ such that
$\dist(\phi(T'', x_q), z_q) < \ep_1$.

Consider hypersurfaces $S_p := \{ x_1 = P_1 y_p \}$, $S_q := \{x_1 =
P_1 y_q\}$. Let us note that Poicare map ${Q: S_q \to S_p}$ is
linear, defined by $Q(x) = \phi(T, x)$ and satisfy
$Q(\{x_2, x_4 = 0\}) \ne \{x_3, x_4 = 0\}$.
Choose $R > 0$ from Lemma \ref{lemCyl}, applied to hypersurface $S_p$, $S_q$, mapping $Q$ and $D = \ep/8$. Note that for some $T_R > 0$ hold the
inequalities
\begin{equation} \label{Cyl4.1.1}\notag
|\phi(t, P_{34}x_p)| < R, \; t < -T_R; \quad  |\phi(t, P_{24}x_q)| < R, \; t > T_R.
\end{equation}
Consider the sets
$$
\Sp^- = \{ \phi(t, P_{34}x_p), \; t < -T_R \}; \quad \Sp^+ = \{ \phi(t, P_{24}x_q), \; t > T_R \}.
$$
Due to Lemma \ref{lem2d} item (ii) sets $\Sp^{\pm}$ satisfy assumptions of Lemma \ref{lemCyl} and hence the sets
$$
C^- = \{x \in S_p: \quad P_{34}x \in \Sp^-, |P_2 x| < D \},
$$
$$
C^+ = \{x \in S_q: \quad P_{24}x \in \Sp^+, |P_3 x| < D\}
$$
satisfy $C^- \cap QC^+ \ne \emptyset$. Let us consider a point
\begin{equation}\label{Text9.0.5}
x_0 \in C^- \cap Q C^+
\end{equation}
and $t_p < -T_R$, $t_q > T_R$, such that
$
P_{34} x_0 = \phi(t_p, P_{34}x_s)$,
$P_{24} Q^{-1}x_0 = \phi(t_q, P_{24}x_u)$.
The following inclusions hold
$$
\phi(-T_Q - T_R - T'', x_0) \in B(2\ep_1, z_q); \quad \phi(-T_Q, x_0) \in B(D,
y_q);
$$
$$
\phi(0, x_0) \in B(D, y_p); \quad \phi(T_R+ T', x_0) \in B(2\ep_1,
z_p).
$$

Inequalities \sref{Pst1.4.1} imply that if $\delta$ is small enough,
then
\begin{equation}\label{Pst1.6.5}
\dist(\phi(t_3 + t, x_0), g(T_p + t)) < \ep/2, \quad t \in [-T, 0].
\end{equation}
Define a reparametrization $h(t)$ as follows:
$$
h(t) = \begin{cases}
h(T_q + T'' + t) = -T_Q-T_R - T'' + t, & t<0,\\
h(T_p + T' + t) = T_R + T' + t, & t>0, \\
h(T_p + t) = t, & t \in [-T, 0], \\
h(t) \; \mbox{increases}, & t \in [T_p, T_p +T'] \cup [T_q + T'',
T_q].
\end{cases}
$$
If $t \geq T_p + T'$, then inequality \sref{Pst1.5.1} implies that
$$
\dist(\phi(h(t), x_0), \phi(t - (T_p + T'), z_p)) < \ep/4;
$$
$$
\dist(\phi(t - T_p, x_p), \phi(t - (T_p + T'), z_p)) < \ep/4.
$$
Hence, if $t \geq T_p + T'$, then
\begin{equation}\label{Pst1.7.1}
\dist(\phi(h(t), x_0), g(t)) < \ep/2.
\end{equation}
For $t \in [T_p, T_p + T']$
the inclusions $\phi(h(t), x_0), g(t) \in B(m, p)$ hold, and
inequality \sref{Pst1.7.1} holds for these $t$ as well.

A similar reasoning shows that inequality \sref{Pst1.7.1} holds for
$t \leq T_q$. If $t \in [T_q, T_p]$, then inequality \sref{Pst1.7.1}
follows from \sref{Pst1.6.5}. This completes the proof in case (P1).
\medskip

Case (P2): $x_p \in \Ws(p)$ and $x_q \notin \Wu(q)$. In this case
the proof uses the same reasoning as in case (P1).
The only difference is that instead of \sref{Text9.0.5} we construct a point $x_0 \in B(D, y_p) \cap W^s_{loc}(p, m)$ such that
$$
\phi(-T - T'', x_0) \in B(2\ep_1, z_q); \quad \phi(-T, x_0) \in B(\ep/8,
y_q).
$$
The construction is straightforward and uses Lemma \ref{lem2d}, item (ii).

\medskip

Case (P3): $x_p \notin \Ws(p)$ and $x_q \in \Wu(q)$. This case is
similar to case (P2).
\medskip

Case (P4): $x_p \in \Ws(p)$ and $x_q \in \Wu(q)$. In this case, we
take $\al$ as the shadowing trajectory; the reparametrization is
constructed similarly to case (P1).
\end{proof}

\begin{remark}\label{remProb}
Proposition \ref{prop3} can be easily generalised in order to prove that $X$ satisfies Lipschitz Oriented Shadowing property. Surprisingly we do not know how to prove Lipschitz analog of Proposition \ref{prop2}.
\end{remark}

\section{Standard Shadowing Property} \label{secst}
Let us show that for small enough $\ep < \min(l, \Delta/2)$ for any $d > 0$ there exists $d$-pseudotrajectory $g(t)$,
which cannot be $\ep$-shadowed.

Put $a_p = p+ (1, 0, 0, 0)$, $a_q = q - (1, 0, 0, 0)$, $e_p
= (0, 0, 0, 1)$ and $e_q = (0, 0, 0, -1)$.
For any $d> 0$ consider pseudotrajectory
$$
g(t) =
\begin{cases}
\phi(t, a_p+ d e_p), & \quad t \geq 0, \\
\phi(t+T_a, a_q + d  e_q), & \quad t \leq -T_a, \\
\phi(t, a_p), & \quad t \in (-T_a, 0).
\end{cases}
$$
Note that for some $L_0 > 0$ map $g(t)$ is $L_0 d$ pseudotrajectory.
Assume that for some $x_0 \in S^2 \times S^2$ and $h(t) \in \Rep(\ep)$
hold the inequalities \sref{eqsh}.
Without loss of generality we can assume that $h(0) = 0$.
Let us consider sets
$$
S_p = \{(1, x_2, x_3, x_4): \; |x_2|, |x_3|,
|x_4| < \Delta\} \subset U_p;
$$
$$
S_q = \{(-1, x_2, x_3, x_4): \; |x_2|, |x_3|,
|x_4| < \Delta\} \subset U_q.
$$
Inequalities \sref{eqsh} imply that $\dist(x_0, a_p + d e_p) <
\ep$, and $\dist(\phi(h(-T_a), x_0), a_q + d e_q) < \ep$. Hence
there exists $L_1 > 0$ and $H_p, H_q \in [-L_1 \ep, L_1 \ep]$ such that points $x_p = \phi(H_p, x_0)$ and $x_q = \phi(h(-T_a) + H_q, x_0)$ satisfy
inclusions $x_p \in S_p$, $x_q \in S_q$.

Inequality \sref{eqsh} implies that for some $L_2 > 0$ the following holds
\begin{equation}\notag
|x_p - a_p|, |x_q - a_q| < L_2 \ep;
\end{equation}
\begin{equation}\label{Text29.1}
\dist(\phi(h(t), x_p), g(t)) < L_2 \ep , \quad t > 0;
\end{equation}
\begin{equation}\notag
\dist(\phi(h(t) - h(-T_a), x_q), g(t)) < L_2 \ep, \quad t \leq -T_a.
\end{equation}
Note that introduced above flow $\psi$
satisfies $\psi(t, (x_3, x_4)) = P_{34}\phi(t, (0, 0, x_3, x_4))$. Hence
inequalities \sref{Text29.1} imply the following
\begin{equation}\label{Text30.1}\notag
\dist(\psi(h(t), P_{34}x_p), \psi(t, (0, d))) < L_2 \ep, \quad
t >0.
\end{equation}
Let us choose $\ep > 0$ satisfying the inequality $L_2 \ep <
l$. Lemma \ref{lem2d} imply that $P_4 x_p > 0$.
Similarly $P_4 x_q < 0$. This contradicts to the
equality $x_p = \phi(T_a,
x_q)$ and (F5). Hence $X \notin \StSh$.

\section{Appendix}
\subsection{Proof of Lemma \ref{lem2d}}
Note that
\begin{equation}\label{31.4}
\psi(t, (r, \varphi)) = (e^{at}r, \varphi+\int_0^t b(e^{a\tau} r)
\dd \tau).
\end{equation}
Item (i). Let us show that $Y \in \Cone(\RR^2)$.
Since $b(r) \in \Cone(0, +\infty)$, it is enough to prove continuity of $D Y(x)$ at $x = 0$. Assume that $\sqrt{x_1^2 + x_2^2} < 2l$.
The following holds:
$$
b'(r) = \frac{1}{r \ln^2 r}, \quad r \in (0, 2l);
$$
$$
\frac{\partial Y_1}{\partial x_1} = a + b'\left(\sqrt{x_1^2 +
x_2^2}\right)\frac{x_1 x_2}{\sqrt{x_1^2 + x_2^2}}; \quad \frac{\partial Y_1}{\partial x_2} =  b'\left(\sqrt{x_1^2 +
x_2^2}\right)\frac{ x^2_2}{\sqrt{x_1^2 + x_2^2}}.
$$
Since
$$
\frac{|x_1 x_2|}{\sqrt{x_1^2 + x_2^2}}, \; \frac{x_2^2}{\sqrt{x_1^2
+ x_2^2}} < \sqrt{x_1^2 + x_2^2}
$$
and $r b'(r) \to 0$ as $r \to 0$, the following holds
$$
\lim_{|x| \to 0}\frac{\partial Y_1}{\partial x_1}(x) = a, \quad \lim_{|x| \to 0} \frac{\partial
Y_1}{\partial x_2}(x) = 0
$$
Arguing similarly for $\frac{\partial Y_2}{\partial
x_1}, \frac{\partial Y_2}{\partial x_2}$ we conclude that
$$
\lim_{|x| \to 0} D Y (x) = \left(
                             \begin{array}{cc}
                               a & 0 \\
                               0 & a \\
                             \end{array}
                           \right).
$$
Note that
$$
\left|
Y(x) - \left(
                             \begin{array}{cc}
                               a & 0 \\
                               0 & a \\
                             \end{array}
                           \right) x \right| =
\left| \left(
                             \begin{array}{cc}
                               0 & b(|x|) \\
                               b(|x|) & a \\
                             \end{array}
                           \right) x
 \right| \leq \frac{|x|}{|\ln(|x|)|},
$$
which implies that
$$
D Y (0) = \left(
                             \begin{array}{cc}
                               a & 0 \\
                               0 & a \\
                             \end{array}
                           \right).
$$
and completes the proof of item (i).

\medskip
Item (ii).
By the equality \sref{31.4} it is enough to show that for $r> 0$, $T_0 < 0$ holds the inequality
\begin{equation}\label{Add2.1}\notag
\int_{-\infty}^{T_0} b(e^{a\tau}r)\dd \tau > 2\pi.
\end{equation}
Without loss of generality we can assume that $r< 2l$. The following holds
$$
\int_{-\infty}^{T_0} b(e^{a\tau}r)\dd \tau = \int_{-\infty}^{T_0}
-\frac{1}{a\tau + \ln r}\dd \tau = \left.-\frac{1}{a} \ln (|a\tau +
\ln r|)\right|_{-\infty}^{T_0} = + \infty.
$$
Item (ii) is proved.

\medskip

Item (iii). Fix $a> 0$. Let $x_0 = (r_0, \varphi_0)$, $x_1 = (r_1, \varphi_1)$ and ${h(t)
\in \Rep(l)}$ satisfy assumptions of the lemma. Let us show that for large enough $K$ and small $l$ holds the inequality
${|\varphi_0 - \varphi_1| < \pi/4}$. Reducing $l$ assume that $K l < 1$. Let us consider $T > 0$ and $\Delta \in \RR$ such that
\begin{equation} \label{31.3.5}
e^{aT}r_0 = Kl, \quad e^{a\Delta}r_0 = r_1.
\end{equation}

Consider points $x_2 = \psi(T, x_0) = (r_2, \varphi_2)$ è $x_3 = \psi(h(T), x_1) = (r_3, \varphi_3)$. Note that $r_2 = Kl$. Inequality
\sref{eq2dstsh} implies
\begin{equation}\label{Text3.1}
\dist(x_2, x_3) < l
\end{equation}
and hence $r_3 \in [(K-1)l, (K+1)l]$. Equality \sref{31.4} implies that
$$
r_3 = e^{ah(T)}r_1, \quad \varphi_2 = \varphi_0 + \int_0^{T}b(e^{a\tau}r_0)\dd \tau, \quad \varphi_3 = \varphi_1 + \int_0^{h(T)}b(e^{a\tau}r_1)\dd \tau.
$$

Relations \sref{eq2dstsh} and \sref{31.3.5} implies
\begin{equation}\label{32.1}
\frac{K}{K+1}e^{a(h(T)+\Delta)}r_0 = \frac{K}{K+1}e^{ah(T)}r_1 < e^{aT}r_0<\frac{K}{K-1}e^{ah(T)}r_1 = \frac{K}{K-1}e^{a(h(T) +
\Delta)}r_0.
\end{equation}
For large enough $K$ inequality \sref{Text3.1} implies that
\begin{equation}\label{32.1.4}
|\varphi_2 - \varphi_3| < \pi/8.
\end{equation}
The following holds
\begin{multline}\label{32.1.5}
\varphi_2 - \varphi_3 = (\varphi_0 - \varphi_1) + \int_0^T
b(e^{a\tau}r_0)\dd \tau - \int_0^{h(T)}b(e^{a\tau}r_1) \dd \tau = \\
= (\varphi_0 - \varphi_1) + \int_0^T b(e^{a\tau}r_0)\dd \tau -
\int_\Delta^{h(T) + \Delta} b(e^{a\tau}r_0)\dd \tau = \\ =
(\varphi_0 - \varphi_1) + \int_0^\Delta b(e^{a\tau}r_0)\dd \tau -
\int_T^{h(T) + \Delta} b(e^{a\tau}r_0)\dd \tau.
\end{multline}
Relations \sref{eq2dstsh} and \sref{31.3.5} imply that
$e^{a(h(T) + \Delta)}r_0 = e^{ah(T)}r_1 > (K-1) l$ and hence
\begin{equation}\label{32.2}
b(e^{a\tau}r_0) = 0, \quad \tau \in [T, h(T) + \Delta].
\end{equation}
Relations \sref{32.1} imply inequalities
$$
\ln \frac{K}{K+1} + a (h(T) + \Delta) < aT< \ln \frac{K}{K-1} + a
(h(T) + \Delta),
$$
and hence
$$
(T - h(T)) - \frac{1}{a} \ln \frac{K}{K-1}< \Delta < (T - h(T)) -
\frac{1}{a}\ln \frac{K}{K+1}.
$$
Since $h(t) \in \Rep(l)$ and $T = (\ln (Kl) - \ln r_0)/a$, holds the inequalities
$$
\frac{1}{a}\left(-l| \ln (K l) - \ln r_0| - \ln
\frac{K}{K-1}\right) < \Delta  < \frac{1}{a}\left(l| \ln (K l)
- \ln r_0| - \ln \frac{K}{K+1}\right).
$$
and hence for large enough $K$ and small $l$ the following holds
\begin{equation}\label{33.1}
|\Delta|< \frac{1}{a}\left(\frac{4}{K-1} - l \ln r_0\right).
\end{equation}
Since $|r_0|< l$, decreasing $l$ and enlarging
$K$ we can assume that $|\Delta|<
\frac{1}{a}\ln 2$. Then for $\tau \in [0, \Delta]$ holds the inequality $e^{a\tau}r_0< 2l$, hence $b(e^{a\tau}r_0) =
1/\ln(e^{a\tau}r_0)$. For small enough $l$ inequality
\sref{33.1} implies that $a|\Delta| < -(\ln r_0)/2$, which implies
$|b(e^{a\tau}r_0)|<2 b(r_0) = -2 / \ln r_0$ and
hence
$$
\left|\int_0^{\Delta} b(e^{a\tau}r_0)\dd \tau\right| < - \frac{2
|\Delta|}{\ln r_0} <  - \frac{2}{a}\left(\frac{4}{K-1} -
l \ln r_0\right)\frac{1}{\ln r_0} < \frac{2}{a}\left(l - \frac{4}{(K-1)\ln l}\right).
$$
For small enough $l$ right hand side of the expression is less than
$\pi/8$. Combining this with relations
\sref{32.1.5}, \sref{32.2} we conclude that
$$
\left|(\varphi_2 - \varphi_3)-(\varphi_0 - \varphi_1)\right| <
\pi/8.
$$
and hence \sref{32.1.4} implies $|\varphi_0 -
\varphi_1| < \pi/4$. Item (iii) is proved.

\subsection{Proof of Lemma \ref{lemCyl}}
Let us fix a linear map $Q$ and a number $D > 0$.
Consider the lines $l_1  \subset S_1$, $l_2 \subset S_2$ defined by ${x_2 =
x_3 =0}$,  $y_2 =
y_3 = 0$ respectively. Note that $Q l_2 \ne l_1$. Let us consider surface $V \subset S_1$ containing $l_1$ and $Ql_2$. Consider a parralelogram $P \subset V$, simmetric with respect to 0 with sides parralel to $l_1$ and $Q
l_2$, satisfying the relation
\begin{equation}\label{Cyl2.1}
P \subset \{|x_1| < D\} \cap Q(\{|y_1|<D\}).
\end{equation}
Let us choose $R>0$, such that the following inclusions hold
\begin{equation}\label{Cyl2.2}
B(R, 0) \cap V \subset P \quad \mbox{and} \quad Q(B(R, 0) \cap
Q^{-1}V) \subset P.
\end{equation}
Let $z_1$ be a point of intersection $\Sp_1$ and the line ${V \cap
\{x_1 = 0\}}$. Condition \sref{Cyl2.2} implies that $z_1 \in P$.
Consider the line $k_1$, containing $z_1$ and parallel to $l_1$.
Inclusion \sref{Cyl2.1} implies that $k_1 \cap P \subset \Cyl_1$.

Similarly let $z_2$ be a point of intersection of $\Sp_2$ and
${V \cap \{y_1 = 0\}}$. Condition \sref{Cyl2.2} implies the inclusion
$Qz_2 \in P$. Let $k_2$ be the line containing $Qz_2$
and parallel to $Ql_2$. Inclusion \sref{Cyl2.1} implies that
$Q^{-1}(k_2 \cap V) \subset \Cyl_2$.

Since $k_1 \nparallel k_2$, there exists a point $z \in k_1 \cap
k_2$. Inclusions ${z_1, z_2 \in P}$ imply that $z \in
P$. Hence $z \in \Cyl_1 \cap Q \Cyl_2$. Lemma
\ref{lemCyl} is proved.

\end{document}